\documentclass[10pt]{amsart}
\usepackage{latexsym}
\usepackage{amsfonts}
\usepackage{graphics}
\usepackage[all]{xy}
\def\Bbb{\mathbb}

\def\eea{\end{eqnarray*}}

\newtheorem{thm}{Theorem}[section]
\newtheorem{prop}[thm]{Proposition}
\newtheorem{cor}[thm]{Corollary}
\newtheorem{lem}[thm]{Lemma}
\newtheorem{conj}[thm]{Conjecture}

\newenvironment{rmk}{\mbox{ }\\{\bf  Remark}\mbox{ }}{
\hfill $\Box$\mbox{}\bigskip}

\begin{document}

\renewcommand{\theequation}{\thesection.\arabic{equation}}

\title{Ricci curvature and monopole classes on $3$-manifolds}

\author{Chanyoung Sung}

\date{\today}

\address{Dept. of Mathematics and Institute for Mathematical Sciences \\
Konkuk University\\
         1 Hwayang-dong, Gwangjin-gu, Seoul, KOREA}
\email{cysung@kias.re.kr}
\thanks{This work was supported by the National Research Foundation of Korea(NRF) grant funded by the Korea government(MEST). (No. 2011-0002791, 2012-0000341)}

\begin{abstract}
We prove an $L^2$-estimate involving Ricci curvature and a harmonic 1-form on a closed oriented Riemannian $3$-manifold admitting a solution of any rescaled
Seiberg-Witten equations. We also give a necessary condition to be a monopole class on some special connected sums.
\end{abstract}
\maketitle

\setcounter{section}{0}
\setcounter{equation}{0}

\section{Introduction}
A second cohomology class is called a \emph{monopole class}
if it arises as the first Chern class of a Spin$^c$ structure for
which the Seiberg-Witten equations
admit a solution for every choice of a Riemannian metric.
It is well-known by LeBrun
\cite{LB1,LB2,LB4} that the existence of a monopole class gives
various curvature estimates of a Riemannian $4$-manifold. These
immediately give corresponding estimates on $3$-manifolds by using the dimensional reduction.
\begin{thm}[\cite{sung}]\label{th0}
Let $(M,g)$ be a smooth closed oriented Riemannian $3$-manifold with
$b_1(M)\geq 1$. Suppose that it admits a solution of the Seiberg-Witten equations
for a Spin$^c$ structure $\mathfrak{s}$. Then
\begin{eqnarray*}
\int_M (s_-)_g^2\ d\mu_g\geq \frac{16\pi^2 |c_1(\mathfrak{s})\cup
[\omega]|^2}{\int_M |\omega|_g^2\ d\mu_g},
\end{eqnarray*}
where  $(s_-)_g$ is $\min (s_g, 0)$ at each point and $s_g$ is the scalar curvature of $g$. Furthermore if the Seiberg-Witten invariant
of $\mathfrak{s}$ is nonzero, then for a nonzero element $\omega$ in $H^1_{DR}(M)$
\begin{eqnarray*}
\int_M |r_g|^2 d\mu_g\geq \frac{8\pi^2 |c_1(\mathfrak{s})\cup
[\omega]|^2}{\int_M |\omega|_g^2\ d\mu_g},
\end{eqnarray*}
where $r_g$ is the Ricci curvature of $g$.
\end{thm}
Here, the Seiberg-Witten invariant in case of $b_1(M)=1$  means that of the chamber for arbitrarily small perturbations.
We conjectured that the above Ricci curvature estimate still holds true when
$c_1(\mathfrak{s})$ is a monopole class.

In this article, we show that it holds true
if $c_1(\mathfrak{s})$ is a {\it strong monopole class} meaning that it is the first Chern class of a  Spin$^c$ structure of $M$ which admits a solution of the rescaled Seiberg-Witten equations of $\mathfrak{s}$ for any rescaling and any Riemannian metric.(See Theorem \ref{th1}.)

Following LeBrun \cite{LB5}, we define the $f$-rescaled Seiberg-Witten equations on a 3-manifold $M$ to be
$$\left\{
\begin{array}{ll} D_A\Phi=0\\
  F_{A}=(\Phi\otimes\Phi^*-\frac{|\Phi|^2}{2}\textrm{Id})f,
\end{array}\right.
$$
where the rescaling factor $f$ is a positive smooth function on $M$.
An obvious but important fact is that if the Seiberg-Witten invariant of $\mathfrak{s}$ is nonzero, then $c_1(\mathfrak{s})$ is a strong monopole class. But it is not  known yet whether every monopole class is a strong monopole class.

In general, it is very difficult to find a monopole class
which has zero Seiberg-Witten invariant. Connected sums of 4-manifolds with $b_2^+>1$ are good candidates. In dimension 4, Bauer and Furuta \cite{bau-fur, bau} devised a new refined invariant of Seiberg-Witten moduli space to prove the existence of a monopole class on some connected sums of K\"ahler surfaces. But it seems that no 3-dimensional example has been found yet. In the final section, we apply our curvature estimates to find a necessary condition to be a monopole class on some special connected sums.

For a brief introduction to the Seiberg-Witten theory, the readers
are referred to \cite{morgan,sung2}.

\section{The Yamabe problem for Modified scalar curvature}

Let $(X,g)$ be a smooth closed oriented Riemannian $4$-manifold and
$W_+$ be the self-dual Weyl curvature. By the modified scalar
curvature we mean $$\frak{S}\equiv s-\sqrt{6}|W_+|.$$ We will denote
the set of $C^{2,\alpha}$ metrics for $\alpha\in (0,1)$ conformal to
$g$ by $[g]$. Assume that there exists a metric in $[g]$ with
nonpositive $\int_X \frak{S}\ d\mu$. Then as observed by Gursky
\cite{gur} and LeBrun \cite{LB3}, the standard proof of the Yamabe
problem \cite{LP} proves that there exists a $C^{2,\alpha}$ metric
in $[g]$ such that $\frak{S}$ is a nonpositive constant. As in
Yamabe problem, it is a ``minimizer" realizing
$$\frak{Y}(X,[g])\equiv\inf_{\tilde{g}\in[g]}
\frac{\int_X\frak{S}_{\tilde{g}}\
d\mu_{\tilde{g}}}{(\textrm{Vol}_{\tilde{g}})^{\frac{1}{2}}}.$$ We
also have
\begin{lem}\label{uniq}
For $r\in [2,\infty]$,
$$\frak{Y}(X,[g])=-\inf_{\tilde{g}\in[g]}(\int_X |\frak{S}_{\tilde{g}}|^r
d\mu_{\tilde{g}})^{\frac{1}{r}}(\textrm{Vol}_{\tilde{g}})^{\frac{1}{2}-\frac{1}{r}}$$
where the infimum is realized only by the minimizer which is unique
up to a constant multiplication.
\end{lem}
\begin{proof}
For $r\in [2,\infty)$, we will use the technique of Besson, Courtois, and Gallot
\cite{BCG}. Let $g$ be a minimizer. Let $\tilde{g}=u^2g$, where $u
: X \rightarrow \Bbb R^+$ is a $C^2$ function. Note that $u$ satisfies the modified Yamabe equation
$${\frak{S}}_{\tilde{g}}u^{3}={\frak{S}}_gu+6\Delta_gu.$$ Therefore
\begin{eqnarray*}
(\int_X |{\frak{S}}_{\tilde{g}}|^r
d\mu_{\tilde{g}})^{\frac{1}{r}}(\textrm{Vol}_{\tilde{g}})^{\frac{1}{2}-\frac{1}{r}}&=&
(\int_X |{\frak{S}}_{\tilde{g}}|^r u^{4}d\mu_{g})^{\frac{1}{r}}
(\int_Xu^4d\mu_{g})^{\frac{1}{2}-\frac{1}{r}}\\
&\geq& \frac{\int_X -{\frak{S}}_{\tilde{g}}u^{2}d\mu_g}{(\int_Xd\mu_g)^{\frac{1}{2}}}\\
&=& \frac{\int_X -({\frak{S}}_{g}+6\frac{1}{u}d^*du)\ d\mu_g} {(\textrm{Vol}_{g})^{\frac{1}{2}}}\\
&=& \frac{\int_X(-{\frak{S}}_{g}+6\frac{|du|^2}{u^2})\ d\mu_g}{(\textrm{Vol}_{g})^{\frac{1}{2}}}\\
&\geq& \frac{\int_X -{\frak{S}}_{g}\
d\mu_g}{(\textrm{Vol}_{g})^{\frac{1}{2}}},
\end{eqnarray*}
where the first inequality is an application of the H\"older
inequality, and the equality holds iff $u$ is a positive constant.
It also follows that any minimizer is a constant multiple of
$g$.

The $L^\infty$ case is an immediate consequence of the other cases because of the inequality
$$
||\frak{S}_{\tilde{g}}||_{L^\infty}(\textrm{Vol}_{\tilde{g}})^{\frac{1}{2}}\geq
(\int_X |{\frak{S}}_{\tilde{g}}|^r d\mu_{\tilde{g}})^{\frac{1}{r}}(\textrm{Vol}_{\tilde{g}})^{\frac{1}{2}-\frac{1}{r}},
$$
whose equality is attained only when $\frak{S}_{\tilde{g}}$ is constant.
\end{proof}

\section{Ricci curvature estimate}

Let us start with the following lemma :
\begin{lem}\label{lem1}
Let $(M,g)$ be a smooth closed oriented Riemannian $3$-manifold and $\mathfrak{s}$ be a Spin$^c$ structure on it. If it admits a solution for a rescaled  Seiberg-Witten equations, then any $C^{2,\alpha}$-metric $\tilde{g}\in [g]$ also has a solution
of the rescaled Seiberg-Witten equations for $\mathfrak{s}$.
\end{lem}
\begin{proof}
We claim that if $(A,\Psi)$ is a solution with respect to $g$, then $(A,
e^{-\varphi}\Psi)$ is a solution with respect to
$\tilde{g}=e^{2\varphi} g$. Mapping an orthonormal frame $\{e_1,e_2,e_3\}$ of $g$ to an orthonormal frame $\{e^{-\varphi}e_1,e^{-\varphi}e_2,e^{-\varphi}e_3\}$ of $\tilde{g}$ gives a global isomorphism of two orthonormal frame bundles and hence a global isometry of the Clifford bundles. Then the identity map between the spinor bundles is an isometry.

For a proof of the Spin$^c$ Dirac
equation, one is referred to \cite{Lawson}, and the curvature
equation is immediate from the fact that
$$|F_A|_{\tilde{g}}=e^{-2\varphi}|F_A|_g=e^{-2\varphi}f|\Psi|_g^2=
f|e^{-\varphi}\Psi|_{\tilde{g}}^2,$$ where $f$ is the rescaling factor.
\end{proof}

\begin{thm}\label{th1}
Let $(M,g)$ be a smooth closed oriented Riemannian $3$-manifold with
$b_1(M)\geq 1$ and $\mathfrak{s}$ be a Spin$^c$ structure on it.
Suppose that it admits a solution for the rescaled  Seiberg-Witten equations for any rescaling.

Then for any
smooth metric $\tilde{g}$ conformal to $g$ and any nonzero
$\omega\in H^1_{DR}(M)$,
\begin{eqnarray*}
\int_M |r_{\tilde{g}}|^2 d\mu_{\tilde{g}}\geq \frac{8\pi^2
|c_1(\mathfrak{s})\cup [\omega]|^2}{\int_M |\omega|_{\tilde{g}}^2\
d\mu_{\tilde{g}}},
\end{eqnarray*}
and the equality holds iff $(M,\tilde{g})$ is a Riemannian submersion onto $S^1$ with totally
geodesic fiber isometric to a compact oriented surface of genus $\geq 1$ with a non-positive constant curvature metric whose volume form is a multiple of $* \omega$ for $\tilde{g}$-harmonic $\omega$, and $[c_1(\mathfrak{s})]$ is a multiple of $[*\omega]$ in $H^2_{DR}(M)$, where $*$ denotes the Hodge star with respect to $\tilde{g}$.
\end{thm}



\begin{proof}
In order to prove the inequality, we may assume $c_1(\mathfrak{s})\ne 0
\in H^2(M,\Bbb R)$. Then the Seiberg-Witten equations have an
irreducible solution for any metric in $[g]$, implying that there
cannot exist a metric in $[g]$ with nonnegative scalar curvature,
and hence there exists a smooth metric in $[g]$ with negative scalar
curvature.

For notational convenience, let $g$ be any smooth metric in $[g]$.
Consider the product metric $g+dt^2$ on $M\times S^1$,where $t\in
[0,1]$ is a global coordinate of $S^1$. By the previous section,
there exists a $C^{2,\alpha}$ metric $\hat{g}\in [g+dt^2]$ which minimizes $\frak{Y}(M\times S^1,[g+dt^2])$ satisfying that $s-\sqrt{6}|W_+|$ is a negative constant.
\begin{lem}
$\hat{g}$ is invariant under the translation along $S^1$-direction.
\end{lem}
\begin{proof}
Let $\hat{g}=(g+dt^2)\psi$ for a positive smooth function $\psi$ on $M\times S^1$, and
we will show $\psi(x,t)=\psi(x,t+c)$ for $(x,t)\in M\times S^1$ for any $c$.

Since $(g+dt^2)\psi(x,t+c)$ is also a minimizer, by Lemma \ref{uniq} there exists a smooth positive function $\varphi$ on $S^1$ such that $$\psi(x,t+c)=\varphi(c)\psi(x,t)$$ for any $(x,t)$. For any $c$,
\begin{eqnarray*}
\int_{M\times S^1}\psi(x,t)\ d\mu_{g+dt^2}&=&\int_{M\times S^1}\psi(x,t+c)\ d\mu_{g+dt^2}\\ &=&\varphi(c)\int_{M\times S^1}\psi(x,t)\ d\mu_{g+dt^2},
\end{eqnarray*}
where the first equality is due to the translation invariance of $dt^2$. Since $\psi>0$, we conclude that $\varphi(c)=1$ for any $c$.
\end{proof}

We write the metric $\hat{g}$ as the warped form $h+f^2 dt^2$ for $f:M\rightarrow \Bbb R^+$ where $h$ is the metric $f^2g$ on $M$. Let
$\{e_1,e_2,e_3,e_4=\frac{\partial}{\partial t}\}$ be a local
orthonormal frame $M\times S^1$ with respect to $h+dt^2$, and
$\{\omega^i|i=1,\cdots,4\}$ its dual coframe. Recall the first Cartan's
structure equations :
$$d\omega^i=-\omega^i_j\wedge \omega^j,$$ where $\omega^i_j$ are the
connection $1$-forms of $h+dt^2$. Obviously $\omega_4^i$ are all zero
for all $i$. Take an orthonormal coframe of $\hat{g}$ as
$\{\omega^1,\omega^2,\omega^3,f\omega^4\}$ and apply the first Cartan's structure equations to
$\{\omega^1,\omega^2,\omega^3,f\omega^4\}$, then one can see that the
connection $1$-forms $\hat{\omega}_j^i$ of $\hat{g}$ are given by
$$\hat{\omega}_j^i=\omega^i_j, \qquad \textrm{for } i,j=1,2,3$$
$$\hat{\omega}_j^4=\frac{\partial f}{\partial e_j}\omega^4 \qquad \textrm{for } j=1,2,3.$$

Let $(A,\Phi)$ be a solution of the $\frac{1}{f}$-rescaled Seiberg-Witten equations for
$\mathfrak{s}$ on $(M,h)$, whose existence is guaranteed by Lemma
\ref{lem1}. Then it is a translation-invariant solution of the
Seiberg-Witten equations for $\mathfrak{s}$ on $(M\times
S^1,h+dt^2)$.

We claim that $(A, \frac{\Phi}{\sqrt{f}})$ is a solution of the unrescaled
Seiberg-Witten equations for $\mathfrak{s}$ on $(M\times
S^1,h+f^2dt^2)$. Let's denote the objects of $\hat{g}=h+f^2dt^2$
corresponding to that of $h+dt^2$ by $\hat{\cdot}$. The Spin$^c$
Dirac equation reads
\begin{eqnarray*}
\hat{D}_{A}(f^{-\frac{1}{2}}\Phi) &=&\sum_{i=1}^4\hat{e}_i\hat{\nabla}_{\hat{e}_i}(f^{-\frac{1}{2}}\Phi)\\
&=&\sum_{i=1}^4\hat{e}_i(\frac{\partial}{\partial\hat{e}_i}(f^{-\frac{1}{2}}\Phi)
+\frac{1}{2}(\sum_{j<k}\hat{\omega}^k_j(\hat{e}_i)\hat{e}_j\hat{e}_k+A(\hat{e}_i))f^{-\frac{1}{2}}\Phi)\\
&=&\sum_{i=1}^3e_i(\frac{\partial}{\partial e_i}(f^{-\frac{1}{2}}\Phi) +\frac{1}{2}(\sum_{j<k\leq
3}{\omega}^k_j(e_i)e_je_k+A(e_i))f^{-\frac{1}{2}}\Phi)\\
& & +\frac{\hat{e}_4}{2}(\sum_{j=1}^3\frac{\partial f}{\partial
e_j}\omega^4(\hat{e}_4)e_j\hat{e}_4)f^{-\frac{1}{2}}\Phi\\
&=&\sum_{i=1}^3e_i(-\frac{f^{-\frac{3}{2}}}{2}\frac{\partial f}{\partial e_i}   \Phi +f^{-\frac{1}{2}}\frac{\partial\Phi}{\partial e_i} +\frac{1}{2}(\sum_{j<k\leq
3}{\omega}^k_j(e_i)e_je_k+A(e_i))f^{-\frac{1}{2}}\Phi)\\
&& +\frac{1}{2}(\sum_{j=1}^3\frac{\partial f}{\partial
e_j}\frac{1}{f}e_j)f^{-\frac{1}{2}}\Phi\\
&=& f^{-\frac{1}{2}}D_A\Phi\\ &=& 0,
\end{eqnarray*}
and the curvature equation reads
\begin{eqnarray*}
F_{A}^{\hat{+}}
&=&\frac{1}{2}(F_{A}+(*_hF_{A})\wedge\hat{\omega}^4)\\ &\simeq&
\frac{1}{2}(F_{A}+(*_hF_{A})\wedge{\omega}^4)\\ &=&
\frac{1}{f}(\Phi\otimes\Phi^*-\frac{|\Phi|^2}{2}\textrm{Id})\\ &=& (f^{-\frac{1}{2}}\Phi)\otimes (f^{-\frac{1}{2}}\Phi)^*-\frac{|f^{-\frac{1}{2}}\Phi|^2}{2}\textrm{Id},
\end{eqnarray*}
where the equivalence in the second line means the identification as
an endomorphism of the plus spinor bundle. (Mapping an orthonormal frame $\{e_1,\cdots,e_4\}$ of $h+dt^2$ to an orthonormal frame $\{e_1,\cdots,e_3,\hat{e}_4\}$ of $h+f^2dt^2$ gives a global isomorphism of two orthonormal frame bundles, inducing a global isometry of the Clifford bundles. Then the identity map between the spinor bundles is an isometry.)

\begin{lem}
$$\int_{M\times S^1} (\frac{2}{3}s_{g+dt^2}-2\sqrt{\frac{2}{3}}|W_+|_{g+dt^2})^2d\mu_{g+dt^2}
\geq 32\pi^2((\pi^*c_1)^+)^2,$$ where $(\pi^*c_1)^+$ is the
self-dual harmonic part of $\pi^*c_1$ with respect to $g+dt^2$, and
$\pi : M\times S^1\rightarrow M$ is the projection map.
\end{lem}
\begin{proof}
This immediately follows from LeBrun's method of Theorem 2.2 in \cite{LB3}. First by using Lemma \ref{uniq},
\begin{eqnarray*}
\int_{M\times S^1} (\frac{2}{3}s_{g+dt^2}-2\sqrt{\frac{2}{3}}|W_+|_{g+dt^2})^2d\mu_{g+dt^2}
&\geq& \int_{M\times S^1} (\frac{2}{3}s_{\hat{g}}-2\sqrt{\frac{2}{3}}|W_+|_{\hat{g}})^2d\mu_{\hat{g}},
\end{eqnarray*}
and the RHS is equal to
\begin{eqnarray}\label{3rdpower}
(\int_{M\times S^1}d\mu_{\hat{g}})^{\frac{1}{3}}
(\int_{M\times S^1} |\frac{2}{3}s_{\hat{g}}-2\sqrt{\frac{2}{3}}|W_+|_{\hat{g}}|^3d\mu_{\hat{g}})^{\frac{2}{3}},
\end{eqnarray}
because $\hat{g}$ has constant $\frac{2}{3}s-2\sqrt{\frac{2}{3}}|W_+|$.
Now we use the fact that $(M\times S^1,\hat{g})$ admits a solution of the unrescaled Seiberg-Witten equations for $\mathfrak{s}$. Combining its Weitzenb\"ock formula with the Weitzenb\"ock formula for the self-dual harmonic 2-forms, we conclude that (\ref{3rdpower}) is greater than or equal to
$32\pi^2((\pi^*c_1)^+)^2.$
\end{proof}
Now using the above lemma, we get
\begin{eqnarray}
\int_M |r_g|^2\ d\mu_g&=&\int_{M\times S^1} |r_{g+dt^2}|^2\
d\mu_{g+dt^2}\nonumber \\
&=& 8\int_{M\times
S^1}(\frac{s_{g+dt^2}^2}{24}+\frac{1}{2}|W_+|_{g+dt^2}^2)\
d\mu_{g+dt^2}\nonumber\\ & &
-8\pi^2(2\chi+3\tau)(M\times S^1)\nonumber\\
&\geq& \frac{1}{2}\int_{M\times S^1} (\frac{2}{3}s_{g+dt^2}-2\sqrt{\frac{2}{3}}|W_+|_{g+dt^2})^2
d\mu_{g+dt^2}-0\label{eq1}\\
&\geq& 16\pi^2((\pi^*c_1)^+)^2 \label{eq2} \\
&\geq&\frac{8\pi^2 |c_1\cup [\omega]|^2}{\int_M |\omega|_g^2\
d\mu_g} \label{eq3},
\end{eqnarray}
where the second equality is due to the $4$-dimensional
Chern-Gauss-Bonnet theorem, and the first inequality is simple
applications of H\"older inequality which was proved in LeBrun
\cite{LB4}.

\begin{lem}
The equality of the theorem statement holds iff $(M\times
S^1,g+dt^2)$ is a K\"ahler manifold of non-positive constant scalar
curvature with the K\"ahler form a multiple of $*\omega+\omega\wedge dt$ for harmonic $\omega$, and $[c_1(\mathfrak{s})]$ is a multiple of
$[*\omega]$ in $H^2_{DR}(M)$.
\end{lem}
\begin{proof}
Let's first consider the case when $[c_1]\ne 0\in H^2_{DR}(M)$. It
is shown in \cite{LB4} that both equalities in (\ref{eq1}) and
(\ref{eq2}) hold iff $g+dt^2$ is a K\"ahler metric of negative
constant scalar curvature with the K\"ahler form a multiple of
$(\pi^*c_1)^+$. The equality in (\ref{eq3}) holds iff
$$\omega=\omega^h=* c_1^h,$$ where $(\cdot)^h$ denotes the
$g$-harmonic part.

When $[c_1]= 0\in H^2_{DR}(M)$, the equality implies that the metric
is Ricci-flat.(In fact, it's a flat manifold $T^3/\Gamma$, because
the dimension is 3.) By the Weitzenb\"ock formula for $1$-forms,
$\omega^h$ is a nonzero parallel $1$-form. Then
$*\omega^h+\omega^h\wedge dt$ is a nonzero parallel $2$-form on
$(M\times S^1,g+dt^2)$, and hence a K\"ahler form with the obvious
complex structure compatible with the orientation. Conversely
suppose that $(M\times S^1,g+dt^2)$ is scalar-flat K\"ahler. Since a
K\"ahler curvature is a (symmetric) section of $\wedge^{1,1}\otimes
\wedge^{1,1}$, on any scalar-flat K\"ahler surface the Riemann
curvature restricted to self-dual two forms must be zero, and hence
so is $W_+$. Then by the $4$-dimensional Chern-Gauss-Bonnet theorem
\begin{align*}
\int_{M\times S^1}|r_{g+dt^2}|^2 d\mu_{g+dt^2}&=\int_{M\times
S^1}(\frac{1}{3}(s_{g+dt^2})^2+4|W_+|_{g+dt^2}^2) d\mu_{g+dt^2}\\
&\ \ \ \ -8\pi^2(2\chi+3\tau)(M\times S^1)\\ &=0,
\end{align*}
giving the equality.
\end{proof}

If the equality holds, we have a parallel splitting of $TM$ by $\omega$ and $*\omega$, each of which gives a manifold of constant scalar curvature of dimension 1 and 2 respectively by the above lemma. Thus the universal cover of $(M,g)$ is isometric to  $\Bbb H^2\times \Bbb R^1$ or $\Bbb R^3$. In the first case, $M$ is a quotient by a discrete subgroup of $PSL(2,\Bbb R)\times \Bbb Z$, and in the second case, $M\times S^1$ is a complex torus or a bielliptic surface by the Enriques-Kodaira classification. Therefore $(M,g)$ is obtained by  identifying two boundaries of $\Sigma\times [0,1]$ by an orientation-preserving isometry of  a compact Riemann surface $(\Sigma, g_{c})$ of genus $\geq 1$ with a constant curvature metric $g_c$. Then $(M,g)$ is locally a Riemannian product of $(\Sigma,g_{c})$ and $S^1$, i.e. a Riemannian submersion onto $S^1$ with totally geodesic fiber $(\Sigma, g_{c})$.

Conversely, suppose that $(M,g)$ is such an oriented Riemannian submersion $\pi: M\rightarrow S^1$. The metric being locally a product,  $(M\times S^1,g+dt^2)$ is a K\"ahler manifold with an obvious complex structure and a K\"ahler form $d\Omega+\pi^*ds\wedge dt$ where $d\Omega$ is the volume form of $(\Sigma, g_{c})$ and $ds$ is the volume form of the base. The scalar curvature of $g$ is just the constant scalar curvature of $g_c$ and the first Chern class of $\Sigma$ is a
multiple of $d\Omega$ which is equal to $*\pi^*ds$. By the above lemma, $(M,g)$ attains the equality, thereby completing the proof.
\end{proof}

\begin{rmk}
As seen in the proof, the condition that $(M,g)$ admits a solution for the rescaled  Seiberg-Witten equations for any rescaling is superfluous. It is enough to suppose that $(M\times S^1, \hat{g})$ has a solution for the unrescaled  Seiberg-Witten equations.

In particular, if  $(M,g)$ has an isometric $G$-action which can be lifted to $\mathfrak{s}$ for a compact Lie group, then it is enough for $M$ to have solutions of the Seiberg-Witten equations for any $G$-invariant metric. In this case, $c_1(\mathfrak{s})$ is called a \emph{$G$-monopole class}, and a $G$-monopole class sometimes exists even when the ordinary Seiberg-Witten invariant vanishes. This will be dealt with in our forthcoming paper (\cite{sung-forth}).
\end{rmk}

\begin{rmk}
It is well-known that on a compact K\"ahler surface $(X, g, \varpi)$ with $b_2^+(X)=1$, the Seiberg-Witten invariant of its canonical Spin$^c$ structure for a small perturbation and a Riemannian metric $\tilde{g}\in [g]$ is equal to $\pm 1$, if $\deg(c_1(X)):=[c_1(X)]\cdot [\varpi]$ is negative (See \cite{morgan}), and also true for any perturbation and Riemannian metric if $c_1(X)$ is a torsion by using the celebrated Taubes's theorem \cite{taubes} and the wall crossing formula \cite{w1,w2}.

\end{rmk}

Now let's discuss some immediate implications of the above theorem. First, by taking $\omega$ to be $*c_1^h(\mathfrak{s})$ where $(\cdot)^h$ denotes the harmonic part, we have
\begin{eqnarray*}
\int_M |r_{\tilde{g}}|^2 d\mu_{\tilde{g}}\geq 8\pi^2
\int_M |c_1^h(\mathfrak{s})|_{\tilde{g}}^2\ d\mu_{\tilde{g}}.
\end{eqnarray*}
More interestingly  we can get a lower bound of $L^2$-norm of a harmonic 1-form on $M$ :
\begin{cor}
Under the same hypothesis as Theorem \ref{th1},  if $\tilde{g}$ is not flat,
\begin{eqnarray*}
(\int_M |\omega|_{\tilde{g}}^2\
d\mu_{\tilde{g}})^{\frac{1}{2}}\geq \frac{2\sqrt{2}\pi
|\alpha\cup [\omega]|}{(\int_M |r_{\tilde{g}}|^2 d\mu_{\tilde{g}})^{\frac{1}{2}}},
\end{eqnarray*}
where $\alpha$ is a convex combination of any two such $c_1(\mathfrak{s})$'s.
\end{cor}

\section{Monopole classes on connected sums}
\setcounter{equation}{0}

Our curvature estimates  provide an easy toolkit in the study of a closed 3-manifold $M$ with a non-torsion monopole.  In \cite{sung}, we derived the inevitability of collapsing when such a manifold has zero Yamabe invariant which implies the existence of a sequence of unit-volume Riemannian metrics $\{g_i\}$ on $M$ satisfying $\inf_i \int_M s_{g_i}^2d\mu_{g_i}=0.$ We also found a necessary condition to be a monopole class in a specific example as follows :
\begin{prop}[\cite{sung}]\label{th2}
Let $M$ be a closed oriented $3$-manifold which fibers over the
circle with a periodic monodromy, and $N$ be a closed oriented
$3$-manifold with $b_1(N)=0$.

Then the rational part of a monopole class, if any, of $M\#
N$ is of the form $m [F]$ for an integer $m$ satisfying $|m|\leq
|\chi(F)|$, where $\chi(F)$ is the Euler characteristic of the fiber $F$.
\end{prop}
(N.B. : In the statement of Theorem 1.4 of \cite{sung}, $b_1(N)=0$ is missing  by mistake.)

We give a generalization of this to connected sums :
\begin{prop}
Let $M_i$ for $i=1,\cdots , n$ be a closed oriented $3$-manifold which fibers over the
circle with a periodic monodromy, and $N$ be any closed oriented $3$-manifold.

Then the rational part of a monopole class, if any,  of $M_1\# \cdots \# M_n\# N$ is of the form $$\beta+\sum_{i=1}^n m_i [F_i]$$ for $\beta\in H^2(N,\Bbb Z)$ and an integer $m_i$ satisfying $|m_i|\leq -\chi(F_i)$, where $\chi(F_i)$ is the Euler characteristic of the fiber $F_i$ in $M_i$.
\end{prop}
\begin{proof}
Let $\alpha$ be a monopole class of $X=M_1\# \cdots \# M_n\# N$.

First if any $F_i$ is a 2-sphere, then the only possibility for $M_i$ is $S^1\times S^2$. Letting $[\omega]$ be the Poincar\'e-dual of $F_i$, we only have to show that it pairs zero with $\alpha$. Let $0< \varepsilon \ll 1$.

 Take a metric of positive scalar curvature on the $M_i$. For the connected sum, we take a small ball on $M_i$, and  a representative $\omega$ of $[\omega]$ to be supported outside of that ball. Then perform the Gromov-Lawson type surgery \cite{GL,sung1, sung11} on it keeping the positivity of scalar curvature to get a compact manifold $M_i'$ with a cylindrical end. And then contract it small enough so that  $$\int_{M_i'} |\omega|^2 d\mu\leq \varepsilon.$$

On the other part of $X$, we put any metric such that it satisfies
\begin{eqnarray}\label{eqn1}
\int (s_-)^2d\mu < 1,
\end{eqnarray}
 and perform the Gromov-Lawson surgery such that the cylindrical end matches with that of the above-made $M_i'$ while still satisfying (\ref{eqn1}).

After gluing these two pieces, we have that
$$\int_{X} |\omega|^2 d\mu\leq \varepsilon\ \ \ \ \textrm{ and } \ \ \ \ \int_X (s_-)^2d\mu < 1.$$
Applying Theorem \ref{th0}, we get $$|4\pi \alpha\cup [\omega]|^2 < \varepsilon,$$ which proves $\alpha\cup [\omega]=0$.

Secondly, let's consider the case of $M_i$ with $\chi(F_i)\leq 0$. Let $[\omega]\in H^1(M_i,\Bbb R)$. By the Mayer-Vietoris principle, $H^1(M_i,\Bbb R)$ is generated by $$\pi_i^*dt,\ \ \  \textrm{ and }\ \ \  \{[\sigma]\in H^1(F_i,\Bbb R) |\ f_i^*[\sigma]= [\sigma]\},$$ where $\pi_i:M_i\rightarrow S^1$ is the projection map and $f_i$ is the monodromy diffeomorphism.

When $[\omega]$ is one of the latter ones, we have to show that it pairs zero with $\alpha$. We can express $\omega$ as
\begin{eqnarray*}
\frac{1}{d_i}\sum_{n=1}^{d_i}(f_i^n)^*\sigma
\end{eqnarray*}
for such  $\sigma$ satisfying $[f^*\sigma]=[\sigma]$, where $d_i$ is the order of $f_i$.

By taking a $f_i$-invariant metric on $F_i$, we can put a locally-product metric on $M_i$ such that $\pi_i$ is a Riemannian submersion with totally geodesic fibers onto a circle of radius $\varepsilon$ and
\begin{eqnarray}\label{eqn2}
\int (s_-)^2d\mu < \varepsilon,
\end{eqnarray}
 We can take a small simply-connected open set $B$ in $F_i$, which is invariant under $f_i$, and take a representative $\sigma$ of the above $[\sigma]\in H^1(F_i,\Bbb R)$ to be supported outside of $B$.

For the connected sum, we perform the Gromov-Lawson surgery on $B\times I(\frac{\varepsilon}{10})\subset M_i$ where $I(\frac{\varepsilon}{10})$ is the interval of length $\frac{\varepsilon}{10}$ to get $M_i'$ with a cylindrical end while still satisfying (\ref{eqn2}). On the other part of $X$, as before we put a metric with a cylindrical end isometric to that of this $M_i'$ while satisfying (\ref{eqn1}).

After gluing two pieces, we have
$$\int_{X} |\omega|^2 d\mu\leq C\varepsilon\ \ \ \ \textrm{ and } \ \ \ \ \int_X (s_-)^2d\mu < 1+\varepsilon,$$ for a constant $C>0$. Hence by Theorem \ref{th0}$$|4\pi \alpha\cup [\omega]|^2 < C\varepsilon(1+\varepsilon),$$ proving $\alpha\cup [\omega]=0$.

Finally when $\omega=\pi_i^* dt$, the adjunction inequality on a 3-manifold $X$, which  can be proved in the same way as 4-manifolds (\cite{kron}) by considering $X\times S^1$  gives  $$|\langle \alpha, [F_i]\rangle | \leq -\chi(F_i),$$ completing the proof.
\end{proof}

\begin{rmk}
As noted, $M_i\times S^1$ for $M_i$ as above admits a K\"ahler metric of constant scalar curvature, and each $M_i$ admits a $d_i$-fold covering space which is $F_i\times S^1$.

For $M_i$ with $F_i=S^2$, one can use the argument of gluing of moduli spaces of Seiberg-Witten equations along  cylindrical ends to prove $4\pi\alpha\cup dt=0.$
\end{rmk}

It seems plausible to conjecture :
\begin{conj}
Let $N_i$ for $i=1,\cdots , n$ be a closed oriented 3-manifold. Then (the rational part of) any monopole class of $N_1\# \cdots \# N_n$ is expressed as $\sum_i^n \alpha_i$ where $\alpha_i$ is a monopole class of $N_i$.
\end{conj}

\bigskip

\noindent{\bf Acknowledgement.} The author would like to warmly
thank Dr. Daewoong Chung for helpful discussions and providing a nice research environment at KIAS.

\end{document}